\documentclass[11pt,reqno]{amsart}
\usepackage{graphicx}
\usepackage{amsmath,amsthm,amssymb}
\usepackage{cite}

\newtheorem{theorem}{Theorem}[section]
\newtheorem{proposition}{Proposition}[section]
\newtheorem{lemma}{Lemma}[section]
\theoremstyle{remark}
\newtheorem{remark}{Remark}[section]
\theoremstyle{corollary}
\newtheorem{corollary}{Corollary}[section]

\numberwithin{equation}{section}

\newcommand{\lap}{\nabla^{2}}

\newcommand{\grad}{\nabla}
\newcommand\K{{\mathcal{K}}}
\newcommand{\dist}{\hspace{.5pt}\text{dist}}

\title{Convergence and Stability of the Inverse Scattering Series for Diffuse Waves}
 
\author{Shari Moskow}
\address{Drexel University} 
\email{moskow@math.drexel.edu} 

\author{John C. Schotland}
\address{University of Pennsylvania} 
\email{schotland@seas.upenn.edu}

\date{\today}


\begin{document}

\begin{abstract}
We analyze the inverse scattering series for diffuse waves in random media. In previous work the inverse series was used to develop fast, direct image reconstruction algorithms in optical tomography. Here we characterize the convergence, stability and approximation error of the series.  
\end{abstract}

\maketitle

\markboth{Shari Moskow and John C. Schotland}
{Convergence and Stability of the Inverse Scattering Series for Diffuse Waves}

\section{Introduction}

The inverse scattering problem (ISP) for diffuse waves consists of recovering the spatially-varying absorption of the interior of a bounded domain from measurements taken on its boundary. The problem has been widely studied in the context of optical tomography---an emerging biomedical imaging modality which uses near-infrared light as a probe of tissue structure and function~\cite{Singer_1990, Arridge_1999}. More generally, diffusion of multiply-scattered waves is a nearly universal feature of wave propagation in random media~\cite{VanR_1999}.

The ISP in optical tomography is usually formulated as a nonlinear optimization problem.   At present, the iterative methods which are used to solve this problem are not well understood mathematically, since error estimates and convergence results are not known. In this paper we will show that, to some extent, it is possible to fill this gap. In particular, we will study the solution to the ISP which arises from inversion of the Born series. In previous work we have utilized such expansions as tools to develop fast, direct image reconstruction algorithms~\cite{Markel_2003}. Here we characterize their convergence, stability and reconstruction error.

Early work on series solutions to the quantum-mechanical inverse backscattering problem was carried out by Jost and Kohn~\cite{Jost_1952}, Moses~\cite{Moses_1956} and Prosser~\cite{Prosser_1969}. There has also been more recent work on the inverse medium problem for acoustic waves~\cite{Devaney_1982, Devaney_2000, Weglein_2003}. However, the procedures employed in these papers are purely formal~\cite{Novikov_1987}. 

The remainder of this paper is organized as follows. In Section~2 we develop the scattering theory of diffuse waves in an inhomogeneous medium---this corresponds to the forward problem of optical tomography.  We then derive various estimates which are later used to study the convergence of the Born series and its inverse. The inversion of the Born series is taken up in Section~3, where we also obtain our main results on the convergence, stability and approximation error of the inverse series. In Section~4 we discuss numerical applications of our results. Finally, in Section~5 we show that our results extend to the case of the ISP for propagating scalar waves. 

\section{Forward Problem}

Let $\Omega$ be a bounded domain in $\mathbb{R}^n$, for $n\ge2$, with a smooth boundary $\partial\Omega$. We consider the propagation of a diffuse wave in an absorbing medium whose energy density $u$ satisfies the time-independent diffusion equation
\begin{equation}
\label{diff_eq}
-\lap u(x) + k^2(1+\eta(x))u(x) = 0 \ , \ \ \ x \in \Omega \ ,
\end{equation}
where  the diffuse wave number $k$ is a positive constant and $\eta(x) \ge -1$ for all $x\in\Omega$. The function $\eta$ is the spatially varying part of the absorption coefficient which is assumed to be supported in a closed ball $B_a$ of radius $a$. The energy density is also taken to obey the boundary condition
\begin{equation}
\label{diff_eq_bc}
u(x) + \ell \nu\cdot\grad u(x) = 0 \ , \ \ \ x \in \partial\Omega \ ,
\end{equation}
where $\nu$ is the unit outward normal to $\partial\Omega$ and the extrapolation length $\ell$ is a nonnegative constant. Note that $k$ and $\eta$ are related to the absorption and reduced scattering coefficients $\mu_a$ and $\mu_s'$  by $k=\sqrt{3\bar\mu_a\mu_s'}$ and $\eta(x)=\delta\mu_a(x)/\bar \mu_a$, where $\delta\mu_a(x)=\bar\mu_a - \mu_a(x)$ and $\bar\mu_a$ is constant~\cite{Markel_2004}.

The forward problem of optical tomography is to determine the energy density $u$ for a given absorption $\eta$. If the medium is illuminated by a point source, the solution to the forward problem is given by the integral equation
\begin{equation}
\label{Lippmann-Schwinger}
u(x)=u_i(x) - k^2\int_\Omega G(x,y)u(y)\eta(y) dy \ , \ \
\ x\in\Omega \ .
\end{equation}
Here $u_i$ is the energy density of the incident diffuse wave which obeys the equation
\begin{equation}
-\lap u_i(x) + k^2 u_i(x) = \delta(x-x_1) \ , \ \ \ x\in\Omega \ , \  x_1\in \partial\Omega 
\end{equation}
and $G$ is the Green's function for the operator $-\lap + k^2$, where $G$ obeys the boundary condition~(\ref{diff_eq_bc}). We may express $G$ as the sum of the fundamental solution 
\begin{equation}
\label{greens_function}
G_0(x,y) = \frac{e^{-k|x-y|}}{4\pi|x-y|} \ , \ \ \  x,y\in \Omega 
\end{equation}
and a solution to the boundary value problem
\begin{eqnarray}
\label{homogeneous}
-\lap F(x) + k^2 F(x) &=& 0 \ ,  \ \ \ x\in \Omega \\
\nonumber
F(x) + \ell \nu\cdot \grad F(x) &=& - G_0(x,y) - \ell \nu\cdot\grad G_0(x,y) \ , \ \ \ x \in \partial \Omega \ ,
\end{eqnarray}
for each $y\in \Omega$. That is, we have $G=G_0 + F$. By~\cite{Gilbarg} Theorem~6.31, $F \in C^2(\bar\Omega)$ when $y$ is in the interior of $\Omega$.

The integral equation~(\ref{Lippmann-Schwinger}) has a unique solution.  If we apply fixed point iteration (beginning with $u_i$), we obtain an infinite series for $u$ of the form
\begin{eqnarray}
\label{forward_series}
\nonumber
u(x)=u_i(x) - k^2\int_\Omega G(x,y)\eta(y)u_i(y) dy \\ + \ k^4 \int_{\Omega\times\Omega}
G(x,y)\eta(y)G(y,y')\eta(y')u_i(y')dydy' + \cdots \ .
\end{eqnarray} 
We will refer to~(\ref{forward_series}) as the Born series and the approximation to $u$ that results from retaining only the linear term in $\eta$ as the Born approximation. 

It will prove useful to express the Born series as a formal power series in tensor powers of $\eta$ of the form
\begin{equation}
\label{phi_def}
\phi = K_1\eta + K_2 \eta\otimes\eta + K_3 \eta\otimes\eta\otimes\eta 
+ \cdots \ ,
\end{equation}
where $\phi=u_i-u$.  Physically, the scattering data $\phi(x_1,x_2)$ is proportional to the change in intensity measured by a point detector at $x_2\in\partial\Omega$ due to a point source at $x_1\in\partial\Omega$~\cite{Markel_2003}. 
Each term in the series is multilinear in $\eta$ and the operator $K_j$ is defined by 
\begin{multline}
\label{def_Kj}
\left(K_jf\right)(x_1,x_2) \\ = (-1)^{j+1} k^{2j}\int_{B_a\times\cdots\times B_a} G(x_1,y_1)G(y_1,y_2)\cdots G(y_{j-1},y_j) 
G(y_j,x_2)f(y_1,\ldots,y_j)dy_1\cdots dy_j \ ,
\end{multline}
for $\ x_1, x_2 \in \partial\Omega$. 
To study the convergence of the series~(\ref{phi_def}) we require an estimate on the $L^\infty$ norm of the operator $K_j$. Let $f\in L^\infty(B_a\times\cdots\times B_a)$.  
Then 
\begin{eqnarray}
\label{eqnarray}
\|K_jf\|_{L^\infty(\partial\Omega\times\partial\Omega)} &= & \sup_{(x_1,x_2)\in\partial\Omega\times\partial\Omega}
\left|(K_jf)(x_1,x_2)\right| \\
\nonumber
&\le& \|f\|_\infty\sup_{(x_1,x_2)\in\partial\Omega\times\partial\Omega}
k^{2j}\int_{B_a\times\cdots\times B_a} |G(x_1,y_1) \cdots G(y_j,x_2)|dy_1\cdots dy_j \ .
\end{eqnarray}
We begin by estimating the above integral for $j=1$,
\begin{eqnarray}
\nonumber
\|K_1\|_\infty & \le& \sup_{(x_1,x_2)\in\partial\Omega\times\partial\Omega}
k^2  \int_{B_a} |G(x_1,y)G(y,x_2)| dy \\
&\le& k^2 |B_a| \sup_{x\in B_a}\sup_{y\in\partial\Omega}|G(x,y)|^2 \ .
\end{eqnarray}
Now, for $j\ge 2$, we take out the first and last factors of $G$ in the integral to obtain
\begin{multline}
\label{Kj_L_infty}
\|K_j\|_\infty 
\le  \sup_{(x_1,x_2)\in \partial\Omega\times\partial\Omega}\sup_{ y_1\in B_a, y_j\in B_a} |G(x_1,y_1)G(y_j,x_2)| \\  \cdot k^{2j}\int_{B_a\times\cdots\times B_a}|G(y_1,y_2) 
\cdots G(y_{j-1},y_j)|dy_1\cdots dy_j \ . 
\end{multline}
We then rewrite this as
\begin{equation}
\|K_j\|_\infty 
\leq \Big(\sup_{x\in B_a}\sup_{y\in\partial\Omega}|G(x,y)| \Big)^2 I_{j-1} \ ,
\end{equation}
where
\begin{equation}
\label{def_I_j}
I_{j-1}= k^{2j}\int_{B_a\times\cdots\times B_a}|G(y_1,y_2)\cdots G(y_{j-1},y_j)|dy_1\cdots dy_j \ .
\end{equation}
Next, we estimate $I_{j-1}$ recursively:
\begin{multline}
I_{j-1} \le\sup_{y_{j-1}\in B_a}k^2\int_{B_a} |G(y_{j-1},y_j)| dy_j \\ \cdot k^{2j-2}\int_{B_a\times\cdots\times B_a}|G(y_1,y_2)\cdots G(y_{j-2},y_{j-1})|dy_1\cdots dy_{j-1} \ ,
\end{multline}
from which it follows that
\begin{equation}
I_{j-1} \le \mu_\infty I_{j-2} \ ,
\end{equation}
where
\begin{equation}
\label{def_mu}
\mu_\infty=\sup_{x\in B_a} k^2 \|G(x,\cdot)\|_{L^1(B_a)} \ .
\end{equation}
We also note that 
\begin{eqnarray}
\label{def_I_1}
I_1 &=& k^4\int_{B_a\times B_a} |G(x,y)|dxdy \\  &\le& k^2 |B_a| \mu_\infty .  \nonumber
\end{eqnarray}
Thus
\begin{equation}
\label{Ij_mu}
I_{j-1} \le k^2 |B_a| \mu_\infty^{j-1} \ .  
\end{equation}
Define 
\begin{equation}
\label{def_nu}
\nu_\infty = k^2 |B_a|\sup_{x\in B_a}\sup_{y\in\partial\Omega}|G(x,y)|^2 \ .
\end{equation}
Note that the smoothness of solutions to~(\ref{homogeneous}) implies that $\mu_\infty$ and $\nu_\infty$ are bounded. Making use of~(\ref{Kj_L_infty}) and (\ref{Ij_mu})
we obtain the following lemma. 
\begin{lemma}
\label{Kj_L_infinity} The operator 
$$ 
K_j:L^\infty(B_a\times\cdots\times B_a) \longrightarrow L^\infty(\partial\Omega\times\partial\Omega)
$$  defined by (\ref{def_Kj}) is bounded and 
$$
\|K_j\|_\infty \le \nu_\infty  \mu_\infty^{j-1} \ ,
$$
where $\mu_\infty$ and $\nu_\infty$ are defined by (\ref{def_mu}) and (\ref{def_nu}), respectively. 
\end{lemma}

We now obtain similar $L^2$ estimates on the norm of $K_j$. Let  
$f\in L^2(B_a\times\cdots\times B_a).$ Then
\begin{equation}
\|K_jf\|_{L^2(\partial\Omega\times\partial\Omega)}^2 = \int_{\partial\Omega\times\partial\Omega} |(K_jf)(x_1,x_2)|^2 dx_1 dx_2 \ .
\end{equation}
From (\ref{def_Kj}) and the Cauchy-Schwarz inequality we have
\begin{eqnarray}
\nonumber
|(K_jf)(x_1,x_2)| & \le & k^{2j} \|f\|_{L^2} \Big(\int_{B_a\times\cdots\times B_a}|G(x_1,y_1)G(y_1,y_2)\cdots G(y_j,x_2)|^2 dy_1\cdots dy_j \Big)^{1/2} \ .
\end{eqnarray}
We begin by estimating $\|K_1\|$:
\begin{eqnarray}
\nonumber
\|K_1\|_2 &\le& k^2 \left(\int_{\partial\Omega\times\partial\Omega}\int_{B_a}
|G(x_1,y)G(y,x_2)|^2 dx_1 dx_2 dy \right)^{1/2} \\
&\le& k^2 |B_a|^{1/2} \sup_{x\in B_a}\|G(x,\cdot)\|_{L^2(\partial\Omega)}^2 \ .
\end{eqnarray}
Now, for $j\ge 2$, we find that
\begin{multline} 
\|K_j\|_2^2  
 \le  \sup_{y_1\in B_a, y_j\in B_a} \int_{\partial\Omega\times\partial\Omega} |G(x_1,y_1)G(x_2,y_j)|^2 dx_1 dx_2 \\ \cdot k^{4j}\int_{B_a\times\cdots\times B_a}  |G(y_1,y_2) 
\cdots G(y_{j-1},y_j)|^2 dy_1\cdots dy_j \ ,
\end{multline}
which yields 
\begin{equation} \label{L2_norm_Kj} \|K_j\|_2^2  
\leq \Big(\sup_{x\in B_a}\|G(x,\cdot)\|_{L^2(\partial\Omega)}\Big)^2 J_{j-1}^2 \ ,
\end{equation}
where
\begin{equation}
\label{def_J_j}
J_{j-1}^2 =   k^{4j}\int_{B_a\times\cdots\times B_a}  
|G(y_1,y_2) \cdots G(y_{j-1},y_j)|^2 dy_1\cdots dy_j \ .
\end{equation}
We now show that $J_{j-1}$ is bounded, which implies that the kernel of $K_j$ is in $L^2(\partial\Omega\times\partial\Omega\times B_a \times \cdots\times B_a)$ and thus $K_j$ is a compact operator. To proceed, we estimate $J_{j-1}$ as follows.
\begin{multline}
J_{j-1}^2 \le \sup_{y_{j-1}\in B_a} k^4 \int_{B_a}|G(y_{j-1},y_j)|^2 dy_j
\ k^{4j-4} \\
\cdot \int_{B_a\times\cdots\times B_a}|G(y_1,y_2)\cdots G(y_{j-2},y_{j-1})|^2 dy_1\cdots dy_{j-1} \ .
\end{multline}
Thus, we find that
\begin{equation}
J_{j-1} \le \mu_2 J_{j-2} \ ,
\end{equation}
where
\begin{equation}
\label{def_mu2}
\mu_2=\sup_{x\in B_a} k^2 \|G(x,\cdot)\|_{L^2(B_a)} \ .
\end{equation}
Noting that $J_1\le k^2 |B_a|^{1/2}\mu_2$, we obtain
\begin{equation}
J_{j-1} \le k^2 |B_a|^{1/2} \mu_2^{j-1} \ .
\end{equation}
Define
\begin{equation}
\label{def_nu2}
\nu_2 = k^2 |B_a|^{1/2} \sup_{x\in B_a}\|G(x,\cdot)\|_{L^2(\partial\Omega)}^2  \ .
\end{equation}
Note that by the smoothness of solutions to~(\ref{homogeneous}) $\mu_2$ and 
$\nu_2$ can be seen to be bounded.
Using~(\ref{L2_norm_Kj}) we have shown the following.
\begin{lemma} 
The operator 
$$ 
K_j:L^2(B_a\times\cdots\times B_a) \longrightarrow L^2(\partial\Omega\times\partial\Omega)
$$  
defined by (\ref{def_Kj}) is bounded and 
$$
\|K_j\|_2 \le \nu_2  \mu_2^{j-1} \ ,
$$
where $\mu_2$ and $\nu_2$ are defined by (\ref{def_mu2}) and (\ref{def_nu2}),  respectively. 
\label{Kj_L_2} 
\end{lemma}

It is possible to interpolate between $L^2$ and $L^\infty$ by making use of the Riesz-Thorin theorem~\cite{Ka}. If $0<\alpha < 1$  then
\begin{equation}
\| K_j \|_{2\over{1-\alpha}}\leq \| K_j\|_2^{1-\alpha}\| K_j\|_\infty^\alpha \ .
\end{equation} 
From Lemmas \ref{Kj_L_infinity} and \ref{Kj_L_2}  we have that
\begin{eqnarray} 
\nonumber 
\| K_j \|_{2\over{1-\alpha}} &\leq& \left( \mu_2^{j-1}\nu_2\right)^{1-\alpha} \left( \mu_\infty^{j-1}\nu_\infty\right)^\alpha \\ 
&=& \left(\mu_2^{1-\alpha}\mu_\infty^\alpha\right)^{j-1} \left(\nu_2^{1-\alpha}\nu_\infty^\alpha\right). 
\end{eqnarray}
We thus obtain the following lemma.
\begin{lemma}
\label{Kj_L_p}  
The operator 
$$ K_j:L^p(B_a\times\cdots\times B_a) \longrightarrow L^p(\partial\Omega\times\partial\Omega)$$ defined by (\ref{def_Kj}) is bounded for $2\leq p\leq\infty$ and $$
\|K_j\|_p\le \nu_p  \mu_p^{j-1} \ ,
$$
where for $2<p<\infty$  
$$
\mu_p=\mu_2^{2\over{p}}\mu_\infty^{1-{2\over{p}}}
$$
and  
$$
\nu_p=\nu_2^{2\over{p}}\nu_\infty^{1-{2\over{p}}} \ .
$$
Here $\mu_2,\nu_2,\mu_\infty$ and $\nu_\infty$ are defined by (\ref{def_mu2}), (\ref{def_nu2}), (\ref{def_mu}) and (\ref{def_nu}), respectively.  
\end{lemma}

To show that the Born series converges in the $L^p$ norm for any $2\leq p\leq\infty$, we majorize the sum 
\begin{equation}
\sum_j\|K_j \eta\otimes\cdots\otimes\eta\|_{L^p(\partial\Omega\times\partial\Omega)}
\end{equation}
by a geometric series:
\begin{eqnarray}
\nonumber
\sum_j\|K_j \eta\otimes\cdots\otimes\eta\|_{L^p(\partial\Omega\times\partial\Omega)} &\le& \sum_j\|K_j\|_p \|\eta\|_{L^p(B_a)}^j \\
&\le& \frac{\nu_p}{\mu_p} \sum_j \left(\mu_p \|\eta\|_{L^p(B_a)}\right)^j \ ,
\end{eqnarray}
which converges if $\mu_p\|\eta\|_{L^p(B_a)} < 1$. When the Born series converges, we may estimate the remainder as follows:
\begin{eqnarray}
\label{Born_error}
\nonumber
\Big\|\phi - \sum_{j=1}^N K_j \eta\otimes\cdots\otimes\eta \Big\|_{L^p(\partial\Omega\times\partial\Omega)} &\le& \sum_{j=N+1}^\infty \|K_j\eta\otimes\cdots\otimes\eta \|_{L^p(\partial\Omega\times\partial\Omega)} \\
\nonumber
&\le& \sum_{j=N+1}^\infty \nu_p \mu_p^{j-1} \|\eta\|_{L^p(B_a)}^j \\
&=& \frac{\nu_p}{\mu_p} \frac{\left(\mu_p\|\eta\|_{L^p(B_a)}\right)^{N+1}}{1-\mu_p \|\eta\|_{L^p(B_a)}} \ .
\end{eqnarray} 
We summarize the above as
\begin{proposition}
\label{Born_convergence} 
If the smallness condition 
$
\|\eta\|_{L^p(B_a)} < 1/\mu_p
$ 
holds, then the Born series (\ref{phi_def}) converges in the $L^p(\partial\Omega\times\partial\Omega)$ norm for $2\leq p\leq\infty$ and  the estimate (\ref{Born_error}) holds. 
\end{proposition}

We note that $L^\infty$ convergence of the Born series for diffuse waves has also been considered in~\cite{Markel_2007}. Corresponding results for the $L^\infty$ convergence of the Born series for acoustic waves have been described by Colton and Kress~\cite{Colton_Kress}.

\section{Inverse Scattering Series}

The inverse scattering problem is to determine the absorption coefficient $\eta$ everywhere within $\Omega$ from measurements of the scattering data $\phi$ on $\partial\Omega$. Towards this end, we make the ansatz that $\eta$ may be expressed as a series in tensor powers of $\phi$ of the form
\begin{equation}
\label{inverse_series}
\eta = \K_1 \phi + \K_2\phi\otimes\phi + \K_3\phi\otimes\phi\otimes\phi + \cdots  \ , 
\end{equation}
where the $\K_j$'s are operators which are to be determined.
To proceed, we substitute the expression { (\ref{phi_def}) for
$\phi$  into  (\ref{inverse_series})}  and equate terms with the same tensor 
power of $\eta$. We thus obtain the relations
\begin{eqnarray}
\K_1K_1 & = & I \ , \label{ident}\\
\K_2K_1 \otimes K_1 + \K_1K_2 & = & 0 \ , \\
\K_3 K_1 \otimes K_1 \otimes K_1 + \K_2K_1\otimes K_2 + \K_2K_2 \otimes K_1 + \K_1K_3 
& = & 0  \ , \\
\sum_{m=1}^{j-1} \K_m \sum_{i_1+\cdots+i_m=j} K_{i_1} \otimes \cdots \otimes K_{i_m} + \K_j K_1 \otimes \cdots \otimes K_1 &=& 0 \  ,
\end{eqnarray}
which may be solved for the $\K_j$'s with the result
\begin{eqnarray}
\K_1 &=& K_1^+ \ , \\
\K_2 &=& -\K_1K_2\K_1\otimes\K_1 \ , \\
\K_3 &=&  -\left(\K_2K_1\otimes K_2 +
\K_2K_2\otimes K_1+\K_1K_3\right)\K_1\otimes\K_1\otimes\K_1 \ , \\
\label{def_K_j}
\K_j &=& - \left(\sum_{m=1}^{j-1} \K_m \sum_{i_1+\cdots+i_m=j} K_{i_1} \otimes \cdots \otimes
K_{i_m}\right) \K_1 \otimes \cdots \otimes \K_1 \ .
\label{def_kappa}
\end{eqnarray}

We will refer to (\ref{inverse_series}) as the inverse scattering series. Here we note several of its properties. First, $K_1^+$ is the regularized pseudoinverse of the operator $K_1$, not a true inverse, so (\ref{ident}) is not satisfied exactly.  The singular value decomposition of the operator $K_1^+$ can be computed analytically for special geometries~\cite{Markel_2004}. Since the operator $\K_1$ is unbounded,  regularization of $K_1^+$ is required to control the ill-posedness of the inverse problem. Second, the coefficients in the inverse series have a recursive structure. The operator $\K_j$ is determined by the coefficients of the Born series $K_1, K_2,\ldots, K_j$. Third, the inverse scattering series can be represented in diagrammatic form as shown in Figure~\ref{fig:diagrams}. A solid line corresponds to a factor of $G$, a wavy line to the incident field, a solid vertex ($\bullet$) to $\K_1\phi$, and the effect of the final application of $\K_1$ is to join the ends of the diagrams. Note that the recursive structure of the series is evident in the diagrammatic expansion which is shown to third order. 

\begin{remark}
Inversion of only the linear term in the Born series is required to compute the inverse series to all orders. Thus an ill-posed nonlinear inverse problem is reduced to an ill-posed linear inverse problem plus a well-posed nonlinear problem, namely the computation of the higher order terms in the series. 
\end{remark}

\begin{figure}[t]
\centering
\includegraphics[width=.7\textwidth,trim=-10 470 -10 -10,keepaspectratio=true]{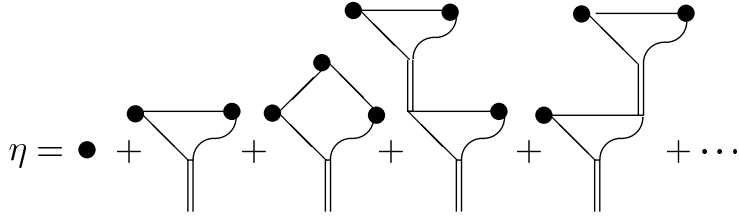} 
\caption{Diagrammatic representation of the inverse scattering series.}
\label{fig:diagrams}
\end{figure}

We now proceed to study the convergence of the inverse series. We begin by developing the necessary estimates on the norm of the operator $\K_j$.  Let $2\le p \le\infty$. Then
\begin{eqnarray}
\nonumber
\|\K_j\|_p &\le& \sum_{m=1}^{j-1}\sum_{i_1+\cdots+i_m=j}\|\K_m\|_p \|K_{i_1}\|_p \cdots \|K_{i_m}\|_p\|\K_1\|_p^j \\
&\le& \|\K_1\|_p^j \sum_{m=1}^{j-1}\sum_{i_1+\cdots+i_m=j} \|\K_m\|_p \nu_p \mu_p^{i_1-1} \cdots
\nu_p \mu_p^{i_m-1}  \ ,
\end{eqnarray}
where we have used Lemma~\ref{Kj_L_p} to obtain the second inequality. Next, we define $\Pi(j,m)$ to be the number of ordered partitions of the integer $j$ into $m$ parts. It can be seen that 
\begin{eqnarray}
\label{def_Pi}
\Pi(j,m) &=& \binom{j-1}{m-1} \ , \\
\sum_{m=1}^{j-1} \Pi(j,m) &=& 2^{j-1} -1 \ .
\end{eqnarray}
Thus in the diagrammatic representation of the inverse series  shown in Figure~\ref{fig:diagrams}, there are $2^{j-1}-1$ diagrams of order $j$ and $\Pi(j,m)$ topologically distinct diagrams of order $j$ and degree $m$ with $m=1,\dots,j-1$.
It follows that 
\begin{eqnarray}
\nonumber
\|\K_j\|_p &\le& \|\K_1\|_p^j \sum_{m=1}^{j-1} \|\K_m\|_p \Pi(j,m)
\nu_p^m \mu_p^{j-m} \\
\nonumber
&\le& \|\K_1\|_p^j \left(\sum_{m=1}^{j-1} \|\K_m\|_p\right) \left(\sum_{m=1}^{j-1} \Pi(j,m)\nu_p^m \mu_p^{j-m}\right) \\
\nonumber
&\le& \nu_p \|\K_1\|_p^j \left(\sum_{m=1}^{j-1} \|\K_{m}\|_p\right)\left( \sum_{m=0}^{j-1} \binom{j-1}{m} \nu_p^m \mu_p^{j-1-m} \right)\\
&=& \nu_p \|\K_1\|_p^j (\mu_p + \nu_p)^{j-1} \sum_{m=1}^{j-1} \|\K_{m}\|_p  \ .
\end{eqnarray}
Thus $\|\K_j\|_p$ is a bounded operator and
\begin{equation}
\label{estimate_K_j}
\|\K_j\|_p \le (\mu_p+\nu_p)^j\|\K_1\|_p^j \sum_{m=1}^{j-1} 
\|\K_{m}\|_p \ .
\end{equation}

The above estimate for $\|\K_j\|_p$ has a recursive structure. It can be seen that 
\begin{equation}
\|\K_j\|_p =   C_j[(\mu_p+\nu_p)\|\K_1\|_p]^j \|\K_1\|_p \ ,
\end{equation}
where, for $j\ge 2$, $C_j$ obeys the recursion relation
\begin{equation}
C_{j+1} = C_{j} + [(\mu_p+\nu_p)\|\K_1\|_p]^{j}C_{j} \ ,  \ \ \  C_2=1 \ .
\end{equation}
It can be seen that
\begin{equation}
C_j = \prod_{m=2}^{j-1} \left(1+[(\mu_p+\nu_p)\|\K_1\|_p]^m \right)  \ .
\end{equation}
Evidently, $C_j$ is bounded for all $j$ since 
\begin{eqnarray}
\nonumber
\ln C_j &\le& \sum_{m=1}^{j-1}\ln \left(1+[(\mu_p+\nu_p)\|\K_1\|_p]^m \right) \\
\nonumber
&\le& \sum_{m=1}^{j-1} [(\mu_p+\nu_p)\|\K_1\|_p]^m \\
&\le& \frac{1}{1-(\mu_p+\nu_p)\|\K_1\|_p} \ ,  
\label{Cbound}
\end{eqnarray}
where the final inequality follows if $(\mu_p+\nu_p)\|\K_1\|_p < 1$. A more refined calculation using the Euler-Maclaurin sum formula gives
\begin{eqnarray}
\ln C_j \lesssim \frac{\text{Li}_2(-(\mu_p+\nu_p)\|\K_1\|_p)}{\ln((\mu_p+\nu_p)\|\K_1\|_p)} + \frac{1}{2}\ln((\mu_p+\nu_p)\|\K_1\|_p) \ , \label{Cbound2}
\end{eqnarray}
where $\text{Li}_2$ is the dilogarithm function. We have shown the following.
\begin{lemma}
\label{nonrecursive}
Let $(\mu_p+\nu_p)\|\K_1\|_p < 1$ and $2\le p \le \infty$ . Then the operator
$$
\K_j : L^p(\partial\Omega\times\cdots\times \partial\Omega) \longrightarrow L^p(B_a)
$$
defined by (\ref{def_kappa}) is bounded  and 
\begin{equation*}
\|\K_j \|_p \le C(\mu_p+\nu_p)^j\|\K_1\|_p^j  \ ,
\end{equation*}
where $C=C(\mu_p,\nu_p,\| \K_1 \|_p)$ is independent of $j$.
\end{lemma}

Note that (\ref{Cbound}) and (\ref{Cbound2}) give explicit bounds for $C$. 

We are now ready to state our main results.

\begin{theorem}[Convergence of the inverse scattering series]
\label{thm1}
The inverse scattering series converges in the $L^p$ norm for $2\le p \le \infty$ if $\|\K_1\|_p < 1/(\mu_p+\nu_p)$ and $\|\K_1 \phi\|_{L^p(B_a)} < 1/(\mu_p+\nu_p)$. Furthermore, the following estimate for the series limit $\tilde\eta$ holds
\begin{eqnarray*}
\Big\|\tilde\eta-\sum_{j=1}^N \K_j \phi\otimes\cdots\otimes\phi\Big\|_{L^p(B_a)} \le C \frac{\left((\mu_p+\nu_p)\|\K_1\|_p\|\phi\|_{L^p(\partial\Omega\times\partial\Omega)}\right)^{N+1}}{1-(\mu_p+\nu_p)\|\K_1\|_p\|\phi\|_{L^p(\partial\Omega\times\partial\Omega)}} \ ,
\end{eqnarray*}
where $C=C(\mu_p,\nu_p, \| \K_1\|_p )$  does not depend on $N$ nor on the scattering data $\phi$.
\end{theorem}

\begin{proof}
The series $\sum_j \K_j \phi\otimes\cdots\otimes\phi$ converges in norm if
\begin{eqnarray}
\label{inv_convergence}
\sum_j \| \K_j \phi\otimes\cdots\otimes\phi\|_{L^p(B_a)} &\le& 
\sum_j \|\K_j\|_p \|\phi\|_{L^p(\partial\Omega\times\partial\Omega)}^j \\
\nonumber
&\le& C \sum_j \left[(\mu_p+\nu_p)\|\K_1\|_p \|\phi\|_{L^p(\partial\Omega\times\partial\Omega)}\right]^{j}  \ ,
\end{eqnarray}
converges, where the last bound follows from Lemma \ref{nonrecursive}. Clearly, the right hand side of (\ref{inv_convergence}) converges when 
\begin{equation}
(\mu_p+\nu_p) \|\K_1\|_p  \|\phi\|_{L^p(\partial\Omega\time\partial\Omega)} < 1 \ .
\end{equation}
To estimate the remainder we consider
\begin{eqnarray}
\nonumber
\Big\|\tilde\eta-\sum_{j=1}^N \K_j \phi\otimes\cdots\otimes\phi\Big\|_{L^p(B_a)} &\le& \sum_{j=N+1}^\infty \|\K_j\phi\otimes\cdots\otimes\phi\|_{L^p(B_a)} \\
&\le& C \sum_{j=N+1}^\infty \left[(\mu_p+\nu_p)\|\K_1\|_p\|\phi\|_{L^p(\partial\Omega\times\partial\Omega)}\right]^j \ , 
\end{eqnarray}
from which the desired result follows.
\end{proof}

We now consider the stability of the limit of the inverse scattering series under perturbations in the scattering data.

\begin{theorem}[Stability]
\label{thm2}
Let  $\|\K_1\|_p < 1/(\mu_p+\nu_p)$ and let $\phi_1$ and $\phi_2$ be scattering data for which $M\|\K_1\|_p  < 1/(\mu_p+\nu_p)$, where
$M=\max{(\|\phi_1\|_p,\| \phi_2\|_p)}$ and $2\le p \le \infty$.
Let $\eta_1$ and $\eta_2$ denote the corresponding limits of the inverse scattering series. Then the following estimate holds
\begin{equation*}
\|\eta_1-\eta_2\|_{L^p(B_a)} < \tilde{C}  \|\phi_1-\phi_2\|_{L^p(\partial\Omega\times\partial\Omega)}  \ ,
\end{equation*}
where $\tilde{C}=\tilde{C}(\mu_p,\nu_p,\| \K_1\|_p,M)$ is a constant which is otherwise independent of $\phi_1$ and $\phi_2$.
\end{theorem}

\begin{proof}
We begin with the estimate
\begin{equation}
\|\eta_1-\eta_2\|_{L^p(B_a)} \le \sum_j \|\K_j(\phi_1\otimes\cdots\otimes\phi_1-
\phi_2\otimes\cdots\otimes\phi_2) \|_{L^p(B_a)} \ .
\end{equation}
Next, we make use of the identity
\begin{eqnarray}
\label{moskow_identity}	
\nonumber
&&\phi_1\otimes\cdots\otimes\phi_1 -
\phi_2\otimes\cdots\otimes\phi_2 =
\psi\otimes\phi_2\otimes\cdots\otimes\phi_2  
+ \phi_1\otimes\psi\otimes\phi_2\otimes\cdots\otimes\phi_2 \\ 
&& + \cdots  + \phi_1\otimes\phi_1\otimes\cdots\otimes\psi\otimes\phi_2 + \phi_1\otimes\phi_1\otimes\cdots\otimes\phi_1\otimes\psi \ ,
\end{eqnarray}
where $\psi=\phi_1-\phi_2$. It follows that
\begin{eqnarray}
\nonumber
\|\eta_1-\eta_2\|_{L^p(B_a)} &\le& \sum_j \sum_{k=1}^j \|\K_j\|_p \|\phi_1\otimes
\cdots\otimes\phi_1\otimes\psi
\otimes\phi_2\otimes\cdots\otimes\phi_2 \|_{L^p} \\
&=& \sum_j j \|\K_j\|_p M^{j-1}\|\psi\|_{L^p(\partial\Omega\times\partial\Omega)} \ ,
\end{eqnarray}
where $\psi$ is in the $k$th position of the tensor product. Using Lemma~\ref{nonrecursive}, we have
\begin{eqnarray}
\nonumber
\|\eta_1-\eta_2\|_{L^p(B_a)} &\le& C \|\K_1\|_p \|\psi\|_{L^p(\partial\Omega\times\partial\Omega)} \sum_j j \left[(\mu_p + \nu_p)\|\K_1\|_p M\right]^j \\
&\le&  \|\K_1\|_p \|\phi_1-\phi_2\|_{L^p(\partial\Omega\times\partial\Omega)} \frac{C}{\left[1-(\mu_p+\nu_p)\|\K_1\|_pM\right]^2} \ .
\end{eqnarray}
The above series converges when $(\mu_p+\nu_p)\|\K_1\|_p M < 1$, which holds by hypothesis. 
\end{proof}

\begin{remark}
It follows from the proof of Theorem~\ref{thm2} that $\tilde C$ is proportional to $\|\K_1\|_p$. Since regularization sets the scale of $\|\K_1\|_p$, it follows that the stability of the nonlinear inverse problem is controlled by the stability of the linear inverse problem.
\end{remark}

The limit of the inverse scattering series does not, in general, coincide with $\eta$. We characterize the approximation error as follows.

\begin{theorem}[Error characterization]
\label{error_estimate}
Suppose that $\| \K_1 \|_p< 1/(\mu_p +\nu_p)$, $\|\K_1\phi\|_{L^p(B_a)} < 1/(\mu_p+\nu_p)$ and $2\le p \le \infty$. Let $\mathcal{M}=\max{( \| \eta\|_{L^p(B_a)} , \| \K_1 K_1 \eta\|_{L^p(B_a)} )}$ and assume that $\mathcal{M} < 1/(\mu_p+\nu_p)$. Then the norm of the difference between the partial sum of the inverse series and the true absorption obeys the estimate
\begin{equation*}
\Big\|\eta-\sum_{j=1}^N\K_j \phi\otimes\cdots\otimes\phi\Big\|_{L^p(B_a)} \le
C \|(I-\K_1 K_1)\eta \|_{L^p(B_a)} + 
\tilde C{[(\mu_p+\nu_p)\|\K_1\|_p\|\phi\|]^N \over{ 1-(\mu_p+\nu_p)\|\K_1\|_p\|\phi\|_p }}\ ,
\end{equation*}
where $C=C(\mu_p,\nu_p,\| \K_1\|_p, \mathcal{M})$ and $\tilde C=\tilde C (\mu_p,\nu_p\| \K_1\|_p)$ are independent of $N$ and $\phi$.
\end{theorem}

\begin{proof}
The hypotheses imply that the regularized inverse scattering series 
\begin{equation}
\label{def_tilde_eta}  
\tilde\eta=\sum_j \K_j \phi\otimes\cdots\otimes\phi 
\end{equation}
converges.  The forward series
\begin{equation}
\phi=\sum_j K_j \eta\otimes\cdots\otimes\eta 
\end{equation}
also converges by hypothesis, so we can substitute it 
into (\ref{def_tilde_eta}) to obtain
\begin{equation}
\tilde\eta = \sum_j \tilde\K_j \eta\otimes\cdots\otimes\eta \ ,
\end{equation}
where 
\begin{equation}
\hspace{-235pt}
\tilde\K_1 = \K_1 K_1 \ ,
\end{equation}
and 
\begin{equation}
\tilde\K_j = \left( \sum_{m=1}^{j-1} \K_m \sum_{i_1+\cdots+i_m=j} K_{i_1} \otimes \cdots \otimes K_{i_m} \right)+ \K_j K_1 \otimes \cdots 
\otimes K_1 \ , 
\end{equation}
for $j\geq 2$. 
From (\ref{def_K_j}) it follows that 
\begin{equation}
\tilde\K_j = \sum_{m=1}^{j-1} \K_m \sum_{i_1+\cdots+i_m=j} K_{i_1} \otimes \cdots \otimes K_{i_m}\left(I- \K_1 K_1 \otimes \cdots \otimes \K_1 K_1 \right) ,
\end{equation}
and so we have
\begin{equation}
\tilde\eta = \K_1 K_1 \eta + \tilde\K_2 \eta\otimes\eta + \cdots \ .
\end{equation}
We thus obtain
\begin{equation}
\eta -\tilde\eta = (I-\K_1 K_1)\eta -\K_1  K_2 \left( \eta\otimes\eta -\K_1 K_1\eta\otimes \K_1 K_1\eta\right)+ \cdots  \ ,
\end{equation}
which yields the estimate
\begin{equation}
\|\eta-\tilde\eta\|_p \le \sum_j\sum_{m=1}^{j-1} \sum_{i_1+\cdots+i_m=j}\|\K_m\|_p \|K_{i_1}\|_p \cdots \|K_{i_m}\|_p \| \eta\otimes\cdots\otimes\eta  - \K_1 K_1\eta \otimes\cdots\otimes\K_1 K_1 \eta \|_p\ .
\end{equation}
Next, we put
\begin{equation} 
\psi=\eta-\K_1 K_1\eta
\end{equation} 
and make use of the identity (\ref{moskow_identity}) to obtain
\begin{equation}
\| \eta\otimes\cdots\otimes\eta  - \K_1 K_1\eta \otimes\cdots\otimes\K_1 K_1 \eta \|_p \leq j \mathcal{M} ^{j-1}
\| \psi\|_p  \ .
\end{equation}
We then have
\begin{eqnarray}
\nonumber
\|\eta-\tilde\eta\|_{L^p(B_a)} &\le& 
\sum_j \sum_{m=1}^{j-1} \sum_{i_1+\cdots+i_m=j}\|\K_m\|_p \|K_{i_1}\|_p \cdots \|K_{i_m}\|_p j \mathcal{M}^{j-1} \|\psi\|_p\nonumber \\
&\le & \sum_j \sum_{m=1}^{j-1} j \mathcal{M}^{j-1} \|\K_m\|_p \Pi(j,m) \nu_p^m \mu_p^{j-m} \|\psi\|_p \ , 
\end{eqnarray}
where we have used Lemma \ref{Kj_L_p} and $\Pi(j,m)$ denotes the number of ordered partitions of $j$ into $m$ parts.  Making use of (\ref{def_Pi}), we have
\begin{eqnarray}
\nonumber 
\|\eta-\tilde\eta\|_p 
&\le &\nu_p \sum_j \|\psi\|_p j \mathcal{M}^{j-1} \left(\sum_{m=1}^{j-1} \|\K_m\|_p \right) \left(\sum_{m=0}^{j-1} \binom{j-1}{m} \nu_p^m \mu_p^{j-1-m} \right) \\
&\le & \|\psi\|_p \sum_j \sum_{m=1}^{j-1} j \mathcal{M}^{j-1}  (\mu_p + \nu_p)^j
\|\K_m\|_p\ .
\end{eqnarray}
We now apply Lemma \ref{nonrecursive} to obtain 
\begin{equation} 
\|\eta-\tilde\eta\|_p 
\le  C \|\psi \|_p\sum_j \sum_{m=1}^{j-1} j \mathcal{M}^{j-1}  (\mu_p + \nu_p)^{m+j} \|\K_1\|_p^m \ ,
\end{equation}
since the  constant $C$ from the lemma is independent of $j$. Performing the sum over $m$, we have
\begin{equation} 
\label{sum_m}
\|\eta-\tilde\eta\|_p 
\le  C \|\psi \|_p\sum_j j \mathcal{M}^{j-1}  (\mu_p + \nu_p)^{j} {  (\mu_p + \nu_p)^j\| \K_1\|_p^j-1\over{(\mu_p + \nu_p)\| \K_1\|_p -1}} \ ,
\end{equation}
which is bounded since $\mathcal{M}(\mu_p+\nu_p)<1$ and $(\mu_p+\nu_p)\| \K_1\|_p < 1$ by hypothesis. Eq.~(\ref{sum_m}) thus becomes 
\begin{equation}
\|\eta-\tilde\eta\|_p 
\le  C \|  (I-\K_1 K_1)\eta  \|_p\ \ ,
\end{equation}
where $C$ is a new constant which depends on $\mu_p,\nu_p,\mathcal{M}$ and $\| \K_1\|_p$. 
Finally, using the triangle inequality, we can account for the error which arises from cutting off the remainder of the series. We thus obtain 
\begin{eqnarray}
\nonumber
& &\Big\|\eta-\sum_{j=1}^N\K_j \phi\otimes\cdots\otimes\phi\Big\|_{L^p(B_a)} \\ 
&\le&
\Big\|\eta-\sum_{j}\tilde\K_j \eta\otimes\cdots\otimes\eta\Big\|_{L^p(B_a)} 
\nonumber
+ \sum_{j=N+1}^\infty \| \K_j \phi\otimes\cdots\otimes\phi\|_{L^p(B_a)} \\
&\le& C \|(I-\K_1 K_1) \eta \|_{L^p(B_a)} 
+ \tilde{C}\frac{\left((\mu_p+\nu_p)\|\K_1\|_p\|\phi\|_p\right)^{N+1}}{1-(\mu_p+\nu_p)\|\K_1\|_p\|\phi\|_p} \ .
\end{eqnarray}
\end{proof}

Evidently, in the above proof, we have shown the following.
\begin{corollary}
Suppose the hypotheses of Theorem~\ref{error_estimate} hold, then the norm of the difference between the sum of the inverse series and the absorption $\eta$ can be bounded by 
\begin{equation*}
\|\eta-\tilde\eta\|_{L^p(B_a)} \le C \|(I-\K_1 K_1) \eta \|_{L^p(B_a)} \ ,
\end{equation*}
where $\tilde\eta$ is the limit of the inverse scattering series and $C=C(\mu_p,\nu_p,\| \K_1\|_p, \mathcal{M})$ is a constant which is otherwise independent of $\eta$.
\end{corollary}

\begin{remark}
We note that, as expected, the above corollary shows that regularization of $\K_1$ creates an error in the reconstruction of $\eta$.  For a fixed regularization, 
the relation $\K_1 K_1=I$ holds on a subspace of $L^p(B_a)$ which, in practice, is finite dimensional. By regularizing $\K_1$ more weakly, the subspace will become larger, eventually approaching all of $L^p$. However, in this instance, the estimate in Theorem~\ref{error_estimate} would not hold since $\|\K_1\|_p$ would become so large that the inverse scattering series would not converge.  Nevertheless, Theorem~\ref{error_estimate} does describe what can be reconstructed exactly, namely those $\eta$ for which $\K_1 K_1 \eta =I$. That is, if we know apriori that $\eta$ belongs to a particular finite-dimensional subspace of $L^p$, we can choose $\K_1$ to be a true inverse on this subspace. Then, if $\| \K_1\|_p $ and $\|\K_1 \phi\|_{L^p} $ are sufficiently small, the inverse series will recover $\eta$ exactly. 
\end{remark}

\section{Numerical Results}

\begin{figure}[t] 
\centering
\includegraphics[width=0.7\textwidth]{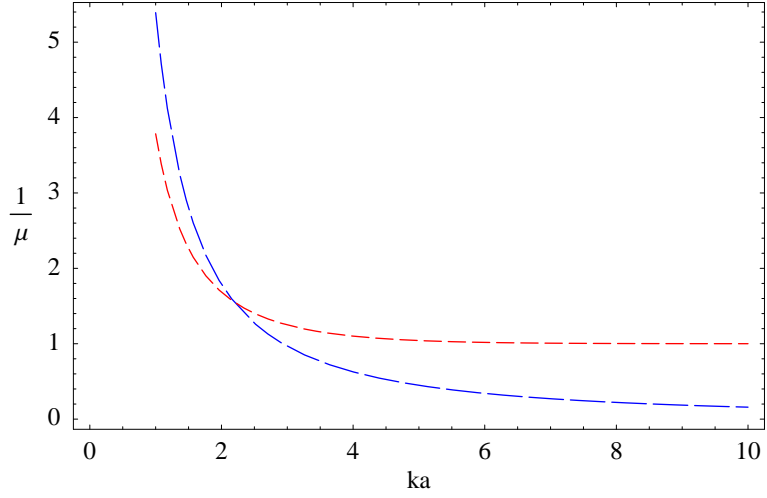} 
\caption{Radius of convergence of the Born series in the $L^2$ (--- --- ---) and $L^\infty$ (-- -- --) norms.}
\label{fig:Born}
\end{figure}

 
It is straightforward to compute the constants $\mu_p$ and $\nu_p$ for $\mathbb{R}^3$. In this case $G(x,y)=G_0(x,y)$ for $x,y\in \Omega$ and the measurements are carried out on $\partial\Omega$ where $\Omega \subset \mathbb{R}^3$. We then have
\begin{eqnarray}
\nonumber
\mu_\infty &=& \frac{k^2}{4\pi}\int_{B_a} \frac{e^{-k|x|}}{|x|}dx \\ 
&=&1-(1+ka)e^{-ka} \ .
\end{eqnarray}
The condition for $L^\infty$ convergence of the Born series becomes
\begin{equation}
\label{eta_infinity_radius}
\|\eta\|_{L^\infty} < \frac{1}{1-(1+ka)e^{-ka}} \ .
\end{equation}
We note that when $ka\ll 1$, we have $\|\eta\|_{L^\infty} \lesssim O(1/(ka)^2)$. In the opposite limit, when $ka\gg 1$, we obtain $\|\eta\|_{L^\infty} \lesssim 1$ and the radius of convergence is asymptotically independent of $ka$. 

For the $L^2$ case we obtain
\begin{eqnarray}
\nonumber
\mu_2 &=& \frac{k^2}{4\pi}\left(\int_{B_a} \frac{e^{-2k|x|}}{|x|^2} dx\right)^{1/2} \\
&=& k^2 e^{-ka/2} \left(\frac{\sinh(ka)}{4\pi k}\right)^{1/2} \ .
\end{eqnarray}
Thus the Born series converges in $L^2$ if
\begin{equation}
\label{eta_2_radius}
\|\eta\|_{L^2} < \frac{e^{ka/2}}{k^2}\left(\frac{4\pi k}{\sinh(ka)}\right)^{1/2} \ .
\end{equation}
When $ka\ll 1$, we have $\|\eta\|_{L^2} \lesssim O(1/(ka)^2)$, which is similar to the $L^\infty$ estimate. However, when $ka\gg 1$ we find that $\|\eta\|_{L^2} \lesssim O(1/(ka)^{3/2})$, which is markedly different than the $L^\infty$ result. The dependence of the radius of convergence on $ka$ is shown in Figure~\ref{fig:Born}.

\begin{figure}[t] 
\centering
\includegraphics[width=0.7\textwidth]{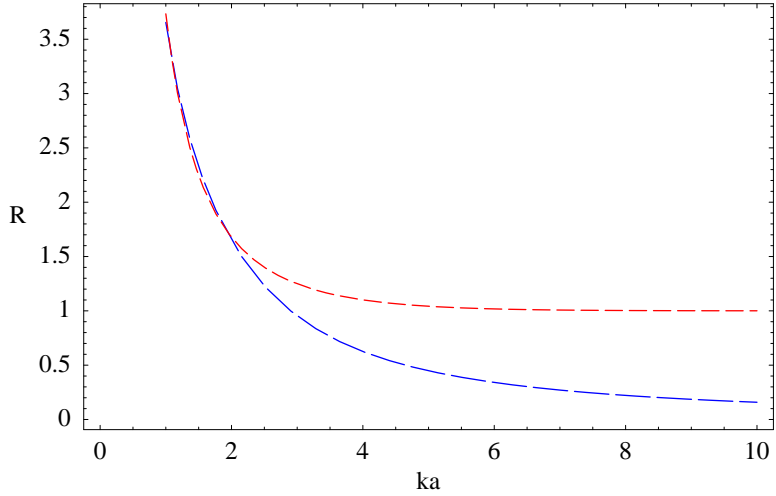} 
\caption{Radius of convergence of the inverse scattering series in the $L^2$ (--- --- ---) and $L^\infty$ (-- -- --) norms.}
\label{fig:inverse_series}
\end{figure}

We now give an application of Theorem~\ref{thm1}. Making use of the Green's function~(\ref{greens_function}), we have
\begin{eqnarray}
\label{nu}
\nu_\infty &\le& k^2 |B_a| \frac{e^{-2k \dist(\partial\Omega, B_a)}}{(4\pi\dist(\partial\Omega, B_a))^2}  \ , \\
\nu_2 &\le& k^2 |\partial\Omega| |B_a|^{1/2}  \frac{e^{-2k \dist(\partial\Omega, B_a)}}{(4\pi\dist(\partial\Omega, B_a))^2} \ .
\end{eqnarray}
Note that $\nu_p$ is exponentially small.
We have computed the radius of convergence $R_p=1/(\mu_p + \nu_p)$ of the inverse scattering series when $\Omega$ is a ball of radius $2a$, where $a$ is the radius of the ball containing the support of $\eta$ and both balls are concentric. The dependence of $R_p$ on $ka$ is shown in Figure~\ref{fig:inverse_series}. It can be seen that $R_\infty \lesssim 1$ and $R_2 \lesssim O(1/(ka)^{3/2})$ when $ka \gg 1$.

\section{Scalar Waves}

In this section we will analyze the inverse scattering series for the case of scalar waves. We will see that although the algebraic structure of the inverse Born series for diffuse waves is similar to that for propagating scalar waves, its analytic structure is different. This distinction is a reflection of the underlying physical difference between the short-range propagation of diffuse waves and the long-range propagation of scalar waves. 

Consider the time-independent wave equation 
\begin{equation}
\label{wave_eq}
\lap u(x) + k^2(1+\eta(x))u(x) = 0 \ , \ \ \  x \in \mathbb{R}^3
\end{equation}
for the propagation of a scalar wave field $u$. 
The wave number $k$ is a positive constant and $\eta(x)$ is assumed to be supported in a closed ball $B_a$ of radius $a$, with $\eta(x) \ge -1$. The field $u$ is taken to obey the outgoing Sommerfeld radiation condition. The solution to (\ref{wave_eq}) is given by the integral equation
\begin{equation}
u(x)=u_i(x) + k^2\int_{\mathbb{R}^3} G(x,y)u(y)\eta(y) dy \ ,
\end{equation}
where $u_i$ is the incident field and the Green's function is given by
\begin{equation}
\label{wave_green_fcn}
G(x,y)=\frac{e^{ik|x-y|}}{4\pi|x-y|} \ .
\end{equation}

\begin{figure}[t] 
\centering
\includegraphics[width=0.7\textwidth]{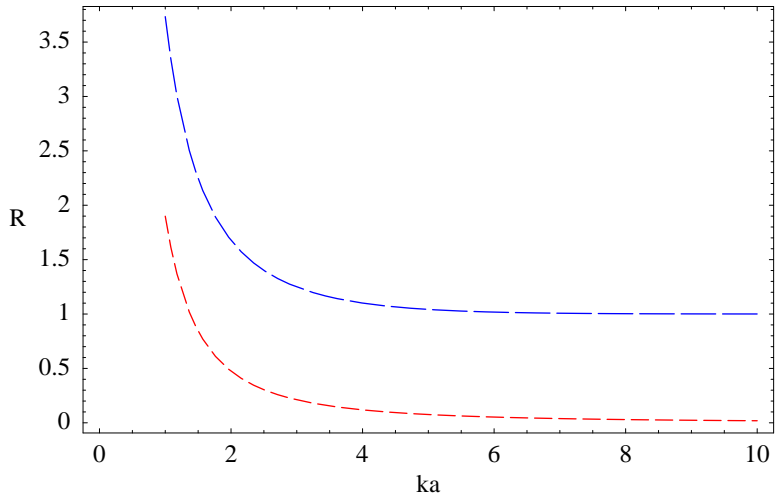} 
\caption{Radius of convergence of the inverse scattering series in the $L^\infty$ norm for diffuse (--- --- ---) and propagating waves (-- -- --).}
\label{fig:comparison}
\end{figure}

The convergence, stability and approximation error of the inverse scattering series corresponding to~(\ref{wave_eq}) is also governed by Theorems~\ref{thm1},     \ref{thm2} and ~\ref{error_estimate}. However, 
the values of the constants $\mu_p$ and $\nu_p$ must be modified. The modifications are a reflection of the physical difference between the oscillatory nature of the Green's function~(\ref{wave_green_fcn}) and the exponentially decaying diffusion Green's function. It can be seen that
\begin{eqnarray}
\mu_{\infty} &=& \frac{1}{2} (ka)^2 \ , \\
\nu_{\infty} &\le&   \frac{k^2 |B_a|}{(4\pi\dist(\partial\Omega, B_a))^2} \ .
\end{eqnarray}
Thus the radius of convergence is $R_\infty = O(1/(ka)^2)$ for $ka\gg 1$. This means that for small scatterers, the inverse scattering series has a large radius of convergence, in contrast to the case of diffuse waves where $R_\infty = O(1)$ even for large scatterers. The dependence of $R_\infty$ on $ka$ is illustrated in Figure 4 for both diffuse and propagating scalar waves. 

\section{Discussion}

A few final remarks are necessary. Numerical evidence suggests that the estimates we have obtained on the convergence of the inverse series are conservative since $\|\K_1\|_p \gg 1$ in practice~\cite{Markel_2003,Markel_2004}. Nevertheless, insight into the structure of the ISP has been obtained. In particular, the inverse series is well suited to the study of waves which do not propagate over large scales such as diffuse waves in random media and evanescent electromagnetic waves in nanoscale systems~\cite{Carney_2003,Carney_2004}. In the latter case, the inverse scattering series has also been shown to be computationally effective~\cite{Panasyuk_2006}. To analyze this problem, our methods must be extended to treat the Maxwell equations. Other areas of future interest include the study of the inverse Bremmer series~\cite{Malcolm_2006}.

\subsection*{Acknowledgements}
We are very grateful to A. Hicks for making this collaboration possible and to C. Epstein for his many useful suggestions. S. Moskow was supported by the NSF grant DMS-0749396. J. Schotland was supported by the NSF grant DMS-0554100 and by the AFOSR grant FA9550-07-1-0096.

\end{document}